\documentclass[a4paper,12pt]{article}
\usepackage[all]{xy} 
\usepackage{amssymb}
\usepackage{epsfig}
\usepackage{amsfonts}
\usepackage{amsmath}
\usepackage{euscript}
\usepackage{amscd}
\usepackage{amsthm}
\DeclareMathAlphabet{\mathpzc}{OT1}{pzc}{m}{it}

\newtheorem{thm}{Theorem}[section]
\newtheorem{lem}[thm]{Lemma}
\newtheorem{prop}[thm]{Proposition} 

\newtheorem{rem}[thm]{Remark}

\newtheorem{defn}[thm]{Definition}

\newcommand{\p}{\mathpzc{p}}

\newcommand{\bQ}{\mathbb Q}

\newcommand{\bC}{\mathbb C}

\newcommand{\A}{\mathbb A}

\newcommand{\Ga}{\mathbb G}

\title{On Separable $\A^2$ and $\A^3$-forms}
\author{Amartya Kumar Dutta{\footnote
{{\small{\it Stat-Math Unit, Indian Statistical Institute,}}
{\small{\it 203 B.T. Road, Kolkata 700 108, India.}}
{\small{\it e-mail : amartya.28@gmail.com}}}},
Neena Gupta{\footnote{{\small{\it Stat-Math Unit, Indian Statistical Institute,}}
{\small{\it 203 B.T. Road, Kolkata 700 108, India.}}
{\small{\it e-mail : neenag@isical.ac.in}}}} and
Animesh Lahiri{\footnote{{\small {\it Swami Vivekananda Research Centre,
Ramakrishna Mission Vidyamandira}}
{\small {\it P.O. Belur Math, Howrah 711202, India.}}
{\small {\it e-mail : 255alahiri@gmail.com }}}}}  

\begin{document}
\date{}
\maketitle

\begin{abstract}
In this paper, we will prove that any $\A^3$-form over a field $k$ of characteristic zero is 
 trivial provided it has a locally nilpotent derivation satisfying certain properties.
We will also show that the result of T. Kambayashi on the triviality of separable $\A^2$-forms over a field $k$ 
extends to $\A^2$-forms over any one-dimensional Noetherian domain containing $\bQ$.

\smallskip
  
\noindent
{\small {{\bf Keywords.} Polynomial Algebra, $\A^n$-form, Separable Extension, Locally Nilpotent Derivation.}}

\noindent
{\small {{\bf AMS Subject classifications (2010)}. Primary: 14R10; Secondary: 13B25; 12F10, 14R25, 13A50}}. 
\end{abstract}

\section{Introduction}
For any commutative ring $R$, we will use the notation $A=R^{[n]}$ to mean that $A$ is a polynomial ring in $n$ variables over $R$.
Now let $k$ be a field with algebraic closure $\bar{k}$ and $A$ be a $k$-algebra. 
We say that $A$ is an $\A^n$-form over $k$ if $A\otimes_k \bar{k}=\bar{k}^{[n]}$.
It is well-known that separable $\A^1$-forms are trivial (i.e., $k^{[1]}$) and that there exist 
non-trivial purely inseparable $\A^1$-forms over fields of positive characteristic.
An extensive study of such algebras was made by T. Asanuma in \cite{A}. 
T. Kambayashi established (\cite{K}) that separable $\A^2$-forms over a field $k$ are trivial. 
Over any field of positive characteristic,
the non-trivial purely inseparable $\A^1$-forms  can  be used to give examples of non-trivial 
$\A^n$-forms  for any integer $n>1$. 
However,  the problem of existence of non-trivial separable $\A^3$-forms 
over a field is still open in general. A few recent partial results on the triviality of 
separable $\A^3$-forms are mentioned in Remark \ref{partial}.

Now let $R$ be a ring containing a field $k$. An $R$-algebra $A$ is said to be an $\A^n$-form over $R$ with respect to $k$ if
$A\otimes_k \bar{k}=(R\otimes_k \bar{k})^{[n]}$, where $\bar{k}$ denotes the algebraic closure of $k$.
In \cite{Du}, A.K. Dutta investigated separable $\A^1$-forms over any ring $R$ containing a field $k$ 
and obtained Theorem \ref{AKD1} quoted below.  He also observed Theorem \ref{AKD2}   
for $\A^2$-forms over any PID containing $\bQ$. 

In this paper, we prove a partial result on separable $\A^3$-forms over a field $k$ (Theorem \ref{main}) and
extend the results on $\A^2$-forms (Theorems \ref{K} and \ref{AKD2})  to any
one-dimensional Noetherian $\bQ$-algebra (Theorem \ref{one-dim}) 
and to any $\bQ$-algebra having a fixed point free locally nilpotent derivation (Theorem \ref{nonregular}). 
After receiving a preprint of our paper, Prof. M. Miyanishi informed us that a part of Theorem \ref{main} 
has also  been obtained recently in \cite{GMM} by a different approach (see Remark \ref{partial} (4) for
a precise statement).


 \section{Preliminaries}

In this section we recall a few definitions and well-known results. 
All rings will be assumed to be commutative containing unity. 
\begin{defn}
{\em An $R$-algebra $A$ is said to be an {\it $\A^{r}$-fibration over $R$} if the following hold:
\begin{enumerate}
\item[\rm(i)] $A$ is finitely generated over $R$.
\item[\rm(ii)] $A$ is flat over $R$.
\item[\rm(iii)] $A\otimes_{R}{k(\p)}=k(\p)^{[r]}$ for every prime ideal $\p$ of $R$.
\end{enumerate}}
\end{defn}

\begin{defn}
{\em Let $k$ be a field of characteristic $p$ $(\geq{0})$ with algebraic closure 
$\bar{k}$ and $R$ a $k$-algebra. An $R$-algebra $A$ is said to be an {\it $\A^n$-form over $R$} (with respect to $k$) 
if $A\otimes_{k}\bar{k}={(R\otimes_{k}{\bar{k}})}^{[n]}$. }
\end{defn}

\begin{defn}
{\em Let $A=R^{[n]}$ and $F\in{A}$. $F$ is said to be a {\it residual coordinate} in $A$ if,
for every prime ideal $\p$ of $R$, $A\otimes_{R}{k(\p)}=k(\p)[\bar{F}]^{[n-1]}$, where $\bar{F}$ denotes the image of
$F$ in $A\otimes_{R}{k(\p)}$.} 
\end{defn}

\begin{defn}
{\em A derivation $D$ on a ring $A$ is said to be a {\it locally nilpotent derivation} if, for each $a \in A$,
there exists an integer $n \ge 0$ (depending on $a$), such that $D^n(a)=0$.  }
\end{defn}

\begin{defn}
{\em We say that a locally nilpotent derivation $D$ on a ring $A$ admits a {\it slice} if there exists $s$ in $A$
for which $D(s)=1$.}
\end{defn}

\begin{defn}
{\em A locally nilpotent derivation $D$ on a ring $A$ is said to be {\it fixed point free} if $(DA)=A$,
where $(DA)$ is the ideal of $A$ generated by $D(A)$.}
\end{defn}
 
 \begin{defn}
{\em Let $R$ be a ring and $D$ a locally nilpotent $R$-derivation on the polynomial ring $A= R^{[n]}$. 
Then the {\it rank of the derivation $D$}, denoted by ${\rm rk}~(D)$, is defined to be the least integer $i$ such that there exist
  $X_1, \dots, X_{n-i} \in {\rm ker ~}D$ satisfying $A= R[X_1, \dots, X_{n-i}]^{[i]}$.} 
\end{defn}

\medskip

\noindent

We first state Kambayashi's theorem (\cite[Theorem 3]{K}) on the triviality of separable $\A^2$-forms over fields.
\begin{thm}\label{K}
Let $k$ and $L$ be fields such that $L$ is separably algebraic over $k$. Suppose $A$ is a  
$k$-algebra such that $A\otimes_{k}L=L^{[2]}$. Then $A=k^{[2]}$.
\end{thm}

\medskip

\noindent

We now state a theorem on separable $\A^1$-forms over rings and a theorem 
on $\A^{2}$-forms over a PID due to Dutta (\cite[Theorem 7 and Remark 8]{Du}). 

\begin{thm}\label{AKD1}
Let $k$ be a field, $L$ a separable field extension of $k$, $R$ a $k$-algebra and 
$A$ an $R$-algebra such that $A\otimes_{k}L$ is isomorphic to the symmetric algebra 
of a finitely generated rank one projective module over 
$R\otimes_{k}L$. 
Then $A$ is isomorphic to the symmetric algebra of a finitely generated rank one projective module over $R$. 
\end{thm}

\begin{thm}\label{AKD2}
Let $k$ be a field of characteristic zero, $R$ a PID containing $k$ and $A$ an $R$-algebra such that 
$A$ is an $\A^{2}$-form over $R$ with respect to $k$. Then $A=R^{[2]}$. 
\end{thm}

\noindent

Next we quote a result on $\A^2$-fibrations due to T. Asanuma and S.M. Bhatwadekar (\cite[Theorem 3.8 and Remark 3.13]{AB}).

\begin{thm}\label{SMB&A}
Let $R$ be a one-dimensional Noetherian $\bQ$-algebra. 
Let $A$ be an $\A^2$-fibration over $R$. 
Then there exists $H\in{A}$ such that $A$ is an $\A^1$-fibration over $R[H]$. 
\end{thm}

\noindent

The following result on residual coordinates was proved by Bhatwadekar and Dutta 
for  Noetherian rings containing $\bQ$ 
(\cite[Theorem 3.2]{BD}) 
and later generalised by A. van den Essen and P. van Rossum for general $\bQ$-algebras (\cite[Theorem 3.4]{RE}).

\begin{thm}\label{res-stable}
Let $R$ be a $\bQ$-algebra, $A= R^{[2]}$ and $F\in{A}$. 
If $F$ is a residual coordinate in $A$ then $A=R[F]^{[1]}$.
\end{thm}

\noindent

Next we state a theorem which follows from a fundamental result in the theory of locally nilpotent derivations (\cite[Corollary 1.26]{ST}).
\begin{thm}\label{slice}
Let $k$ be a field of characteristic zero, $A$ a $k$-algebra, $D$ a locally nilpotent derivation on $A$ and $B:={\rm ker~} D$. Then the following are equivalent:
\begin{enumerate}
\item[\rm (1)] $D$ admits a slice $s$.
\item[\rm (2)] $A=B[s]=B^{[1]}$ and $D=\dfrac{d}{ds}$ on $A$.
\item[\rm (3)] $D(A)=A$.
\end{enumerate} 
\end{thm}

\noindent

The following rigidity theorem is due to D. Daigle (\cite[Theorem $2.5$]{Da}). 

\begin{thm}\label{Darigidity}
Let $k$ be a field of characteristic zero and $D$ be a locally nilpotent derivation on the polynomial ring $A=k^{[3]}$ with 
${\rm rk~}(D)=2$.
Let $X,W\in{{\rm ker~} D}$ be such that $A=k[X]^{[2]}= k[W]^{[2]}$. Then $k[X]=k[W]$. 
\end{thm}

\noindent

The following result on fixed point free locally nilpotent derivations was obtained by
Bhatwadekar and Dutta (\cite[Theorem 4.7]{BD97}) for any Noetherian $\bQ$-algebra 
and later generalized to any $\bQ$-algebra by J. Berson, A. van den Essen and S. Maubach
(\cite[Theorem 3.5]{BEM}; \cite[Theorem 4.15]{ST}).

\begin{thm}\label{BD}
Let $R$ be a $\bQ$-algebra, $A=R[X,Y]=R^{[2]}$, $D$ a fixed point free locally nilpotent $R$-derivation of $A$
and $B={\rm ker}~D$. Then $D$ admits a slice, $B=R^{[1]}$ and $A=B^{[1]}$.
\end{thm}

 \begin{rem}
{\em A fixed point free locally nilpotent derivation on $k[X,Y,Z]$ has a slice (\cite{Kal}). 
But a fixed point free locally nilpotent $R$-derivation on $R[X,Y,Z]$ need not have a slice even if $R$ is a PID (\cite[Example 5.6]{BNL}).}
\end{rem}

 
\section{Main results}

In this section we will prove our main results.
Note that if $k$ is a field of characteristic zero,  $A$ a $k$-algebra and $L$ a field extension of $k$, then any $k$-linear locally 
nilpotent derivation $D$ on $A$ can be extended to a locally nilpotent derivation $D \otimes 1_L$ on $A\otimes_k L$
such that $(D \otimes 1_L)(a\otimes\lambda)=D(a)\otimes\lambda$ for all $a\in{A}$ and $\lambda\in{L}$.
We will first establish our main theorem on $\A^{3}$-forms over $k$ (Theorem \ref{main}).
We begin with a special case of this result which holds for $\A^{3}$-forms over a PID $R$ with respect to $k$.     

\begin{prop}\label{prop}
Let $k$ be a field of characteristic zero with algebraic closure $\bar{k}$, 
$R$ a PID containing $k$ and $A$ be an $\A^3$-form over $R$ with respect to $k$. Suppose that there exists an $R$-linear locally nilpotent derivation $D$ on $A$ 
such that ${\rm rk~}(D\otimes 1_{\bar{k}})=1$. Then $A=R^{[3]}$.  
\end{prop}

\begin{proof}
Since $A$ is an $\A^3$-form over $R$ with respect to $k$, there exists a finite extension $L$ over $k$  
such that $A\otimes_k L= (R\otimes_k L)^{[3]}$ and ${\rm rk~}(D\otimes_k{1_L})=1$.
Let $B={\rm ker~} D$. Set $\bar{R}:=R\otimes_{k}L$, $\bar{A}:=A\otimes_{k}L$, $\bar{B}: =B\otimes_{k}L$ and $\bar{D}:= D\otimes 1_L$. 
Then $\bar{A}=\bar{R}^{[3]}$ and ${\rm ker~} \bar{D}=\bar{B}$.
\smallskip
Since ${\rm rk}(\bar{D})=1$, we have $\bar{A}=\bar{B}^{[1]}$ and $\bar{B}=\bar{R}^{[2]}$. Hence, $B=R^{[2]}$ by Theorem \ref{AKD2}. 
As ${\rm Pic} (B)$ is trivial,  $A=B^{[1]}$  by Theorem \ref{AKD1}. Thus,  $A=R^{[3]}$.
\end{proof}

We now prove our main result on $\A^3$-forms.

\begin{thm}\label{main}
Let $k$ be a field of characteristic zero with algebraic closure $\bar{k}$ and $A$ be an $\A^{3}$-form over $k$. 
Suppose that there exists a $k$-linear locally nilpotent derivation $D$ on $A$ 
such that ${\rm rk~}(D\otimes 1_{\bar{k}})\leq{2}$. Then $A=k^{[3]}$.
\end{thm}
\begin{proof} 
Since $A$ is an $\A^3$-form over $k$, there exists a finite Galois extension $L$ over $k$ with Galois group $G$ 
such that $A\otimes_k{L}=L^{[3]}$ and ${\rm rk~}(D\otimes_k{1_L})\leq{2}$.
Let $B={\rm ker~} D$. Set $\bar{A}:=A\otimes_{k}L$,  $\bar{B}: =B\otimes_{k}L$ and $\bar{D}:= D\otimes 1_L$. Then $\bar{A}=L^{[3]}$ and 
${\rm ker~} \bar{D}=\bar{B}$.

\smallskip
If ${\rm rk}(\bar{D})=1$, then $A=k^{[3]}$ by Proposition \ref{prop} (taking $R=k$).

\smallskip
We now consider the case ${\rm rk~}(\bar{D})=2$. 
We then have  $X\in{\bar{B}}$ such that $\bar{A}=L[X]^{[2]}$. We show that there exists $W\in{L[X]\cap{A}}$ such that $L[X]=L[W]$.

\smallskip
We identify $A$ with its image in $\bar{A}$ under the map $a\rightarrow{a\otimes1}$. 
Any $\sigma\in{G}$ can be extended to an $A$-automorphism of $\bar{A}$ by defining 
$\sigma(a\otimes l)=a\otimes \sigma{(l)}$, for all $a\in{A}$ and $l\in{L}$. Let 
\begin{equation*}
X=1\otimes l_0+e_1\otimes l_1+\dots+e_r\otimes l_r
\end{equation*}
where $1,e_1,\dots,e_r$ form a part of a $k$-basis of $A$ and $l_i$'s are in $L$. Since the bilinear map $L\times{L}\longrightarrow{k}$ given by $(x,y)\mapsto{Tr(xy)}$ is non-degenerate (where $Tr(a):=Trace(a)$ for all $a$ in $L$), replacing $X$ by $lX$ (for some $l\in{L})$ if necessary we can assume that $Tr(l_i)\neq{0}$ for some $i\geq{1}$. Thus
\begin{equation*}
W:=\underset{\sigma\in{G}}{\sum}\sigma(X)=1\otimes{Tr(l_0)}+e_1\otimes{Tr(l_1)}+\dots+e_r\otimes{Tr(l_r)}
\end{equation*} 
is an element of $A\setminus k$. Note that $\sigma\bar{D}=\bar{D}\sigma$ and hence $\sigma{(X)}\in{\bar{B}}$. Since $\sigma$ is an automorphism of $\bar{A}$, by Theorem \ref{Darigidity}, $L[X]=L[\sigma{(X)}]$. Hence $\sigma(X)$ is linear in $X$ for each $\sigma$ and hence $deg_{X}W\leq{1}$. 
But as $B\cap{L}=k$, $W\notin{L}$, so that $deg_{X}{W}=1$ which implies $L[X]=L[W]$. 

So $\bar{A}=L[W]^{[2]}=(k[W]\otimes_{k}L)^{[2]}$. By Theorem \ref{AKD2}, we get $A=k[W]^{[2]}$.   
\end{proof}

\begin{rem}\label{partial} 
{\em Let $k$ be a field of characteristic zero with algebraic closure $\bar{k}$ and $A$ an $\A^3$-form over $k$.
We record below a few other results on the triviality of $A$.

(1) D. Daigle and S. Kaliman have proved (\cite[Corollary 3.3]{DK}) that if $A$ admits 
a fixed point free locally nilpotent derivation $D$, then $A=k^{[3]}$.

(2) Daigle and Kaliman have also proved (\cite[Proposition 4.9]{DK}) that if $A$ contains an element $f$
which is a coordinate of $A \otimes_k \bar{k}$, then $A=k^{[3]}$ and $f$  is a coordinate of $A$.
 
(3) M. Koras and P. Russell have proved (\cite[Theorem C]{RK}) that if $A$  admits 
an effective action of a reductive algebraic $k$-group of positive dimension, then $A=k^{[3]}$. 

(4) Recently, R.V. Gurjar, K. Masuda and M. Miyanishi have shown (\cite{GMM}) that $A=k^{[3]}$ if $A$ admits
either a fixed point free locally nilpotent derivation or a non-confluent action of a unipotent group of dimension two.
Their results give an alternative approach to Theorem \ref{main} for the case ${\rm rk~}(\bar{D})=1$.
}
\end{rem}

We now extend Theorems \ref{K} and \ref{AKD2} to more general rings.
For convenience, we first record a few easy lemmas.

\begin{lem}\label{residual}
Let $R$ be a ring containing $\bQ$ and $A=R^{[2]}$. If $H\in A$ 
is such that $A$ is an $\A^{1}$-fibration over $R[H]$, then $A=R[H]^{[1]}$.
\end{lem}

\begin{proof}
Let $\p$ be a prime ideal of $R$ and let $\bar{H}$ denote the image of $H$ in $A\otimes_{R}{k(\p)}$.
Then $A\otimes_{R}{k(\p)}$ is an $\A^{1}$-fibration over the PID $k(\p)[\bar{H}]$
and hence  $A\otimes_{R}{k(\p)}=k(\p)[\bar{H}]^{[1]}$. 
Thus, $H$ is a residual coordinate of $A$. Hence, by Theorem \ref{res-stable}, $A=R[H]^{[1]}$. 
\end{proof}

We now observe that Theorem \ref{K} extends to 
separable $\A^2$-forms over a field $K$ with respect to a subfield $k$. 

\begin{lem}\label{over field}
Let $k$ be a field and $K$ a field extension of $k$. 
Let $A$ be a $K$-algebra such that $A\otimes_k{L}={(K\otimes_{k}L)}^{[2]}$, for some finite separable field extension $L$ of $k$. Then $A=K^{[2]}$. 
\end{lem}
\begin{proof}
By hypothesis, we have $A\otimes_{K}(K\otimes_{k}{L})={(K\otimes_{k}{L})}^{[2]}$. Since $L$ over $k$ is a finite separable extension, $K\otimes_{k}{L}$ is a finite direct product of separable extensions $L_i$ over $K$. Hence, we have $A\otimes_{K}{L_i}={L_i}^{[2]}$ (for each $i$), which implies $A=K^{[2]}$ by Theorem \ref{K}.
\end{proof}

We now show that $\A^2$-forms are $\A^2$-fibrations.

\begin{lem}\label{fibration}
Let $k$ be a field of characteristic zero, $R$ be a $k$-algebra 
and $A$ be an $R$-algebra. Let $A$ be an $\A^2$-form over $R$ with respect to $k$. 
Then $A$ is an $\A^2$-fibration over $R$.
\end{lem}

\begin{proof}
Let $A\otimes_{k}{\bar{k}}={(R\otimes_{k}{\bar{k}})}[X,Y]$, where $\bar{k}$ is an algebraic closure of $k$. 
Let $X=\underset{i=0}{\overset{n}\sum}{a_i\otimes{{\lambda}_i}}$ and  $Y=\underset{i=0}{\overset{m}\sum}{b_i\otimes{{\mu}_i}}$, 
where $a_i,b_i\in{A}$ and ${\lambda}_i,{\mu}_i\in{\bar{k}}$. Then $R[a_1,\dots,a_n,b_1,\dots,b_m]\subseteq{A}$ and the induced map  $R[a_1,\dots,a_n,b_1,\dots,b_m]\otimes_{k}{\bar{k}}\longrightarrow{A\otimes_{k}{\bar{k}}}$ is an isomorphism. Hence $\bar{k}$ being faithfully flat over $k$, we have $A=R[a_1,\dots,a_n,b_1,\dots,b_m]$. Thus $A$ is a finitely generated $R$-algebra. 
Again, as $A\otimes_{k}{\bar{k}}$ is faithfully flat over $R\otimes_{k}{\bar{k}}$ and 
$\bar{k}$ is faithfully flat over $k$, $A$ is flat over $R$. Now it suffices to show 
$A\otimes_{R}{k(\p)}=k(\p)^{[2]}$, for each prime ideal $\p$ of $R$. 

Let $\p$ be an arbitrary prime ideal of $R$. By hypothesis there exists a 
finite separable extension $L$ of $k$ such that 
$A\otimes_{k}{L}=(R\otimes_{k}{L})^{[2]}$. Hence, 
$k(\p)\otimes_{R}({A}\otimes_k{L}) =k(\p)\otimes_{R}(R\otimes_{k}{L})^{[2]} =(k(\p)\otimes_{k}{L})^{[2]}$. 
Hence by Lemma \ref{over field}, $A\otimes_{R}{k(\p)}=k(\p)^{[2]}$. 

Thus, $A$ is an $\A^2$-fibration over $R$.
\end{proof}


\noindent
We now extend Theorems \ref{K} and \ref{AKD2} to any one-dimensional Noetherian ring containing a field of characteristic zero.

\begin{thm}\label{one-dim}
Let $k$ be a field of characteristic zero and $R$ a one-dimensional Noetherian $k$-algebra. 
If $A$ is an $\A^2$-form over $R$ with respect to $k$,
then there exists a finitely generated rank one projective $R$-module $Q$ such that $A \cong ({\rm Sym}_{R}(Q))^{[1]}$.   
\end{thm}
\begin{proof}
By Lemma \ref{fibration}, $A$ is an $\A^2$-fibration over $R$ and hence by Theorem   \ref{SMB&A}, 
there exists $H\in{A}$ such that $A$ is an $\A^1$-fibration over $R[H]$. 
Let $\bar{k}$ be an algebraic closure of $k$, 
$\bar{A}:=A\otimes{\bar{k}}$ and $\bar{R}:=R\otimes{\bar{k}}$. 
Since $\bar{A}=\bar{R}^{[2]}$ and $\bar{A}$ is an $\A^1$-fibration over $\bar{R}[H]$, 
we have $\bar{A}=\bar{R}[H]^{[1]}$ by Lemma \ref{residual}. 
Thus by Theorem \ref{AKD1}, $A \cong {\rm Sym}_{R[H]}({Q_1})$, for some finitely generated rank one projective $R[H]$-module $Q_1$. 
Set $R_{\text{red}}:= R/nil(R)$. Now
$$
A/nil(R)A \cong {\rm Sym}_{R[H]}({Q_1})\otimes_R R_{\text{red}}={\rm Sym}_{R_{\text{red}}[H]}({Q_1}\otimes_R R_{\text{red}}).
$$
Now by (\cite{I}, Section $2$, Lemma $1$), there exists a finitely generated rank one projective   
$R_{\text{red}}$-module $Q'$ such that $Q_1\otimes_R (R_{\text{red}})=Q'\otimes_{R_{\text{red}}} R_{\text{red}}[H]$. 
Thus, 
$$
A/nil(R)A \cong {{\rm Sym}_{R_{\text{red}}}{(Q')}}\otimes_{R_{\text{red}}}{R_{\text{red}}[H]}=
({\rm Sym}_{R_{\text{red}}}(Q'))^{[1]}.
$$
Now by \cite[Proposition 2.3.5]{IR}, there exists a finitely generated rank one projective $R$-module $Q$ such that
$Q\otimes_R R_{\text{red}} =Q'$ and hence $A= ({\rm Sym}_{R}(Q))^{[1]}$. 
\end{proof}

The following result shows that under the additional hypothesis that 
$A$ has a fixed point free locally nilpotent $R$-derivation, Theorem \ref{one-dim} can be extended to any ring containing a field of characteristic zero. 

\begin{thm}\label{nonregular}
Let $k$ be a field of characteristic zero, $R$ a ring containing $k$ and $A$ be an $\A^{2}$-form over $R$ with respect to $k$. 
Suppose $A$ has a fixed point free locally nilpotent $R$-derivation.
Then there exists a  finitely generated rank one projective $R$-module $Q$ such that $A \cong ({\rm Sym}_{R}(Q))^{[1]}$.  
\end{thm} 
\begin{proof}
Let $L$ be a finite extension of $k$ such that $A\otimes_k L= (R\otimes_k L)^{[2]}$. 
Let $D$ be a fixed point free locally nilpotent $R$-derivation of $A$ and $B= {\rm ker}~D$.
Set $\bar{R}:=R\otimes_k L$, $\bar{A}:=A\otimes_k L$, $\bar{B}:=B\otimes_k L$ and $\bar{D}:=D\otimes 1_L$. Then 
$\bar{A}= \bar{R}^{[2]}$, ${\rm ker~} \bar{D} = \bar{B}$ and
$\bar{D}$ is a fixed point free locally nilpotent derivation of $\bar{A}$. Hence, by Theorem \ref{BD}, $\bar{D}$ has a slice and 
$\bar{B}= \bar{R}^{[1]}$. 
Now, by Theorem \ref{slice}, $D(\bar{A})=\bar{A}$. Thus, $D(A)\otimes_{k}L=D(\bar{A})=\bar{A}=A\otimes_{k}L$. Hence, by faithful flatness of $L$ over $k$, $D(A)=A$. So ${A}= {B}^{[1]}$ by Theorem \ref{slice}. Since $\bar{B}= \bar{R}^{[1]}$, by Theorem \ref{AKD1}, 
$B={\rm Sym}_R(Q)$
and hence $A=({\rm Sym}_{R}(Q))^{[1]}$,
for some finitely generated rank one projective $R$-module $Q$. 
\end{proof}

\begin{rem}
\em {M.E. Kahoui and M. Ouali have shown (\cite[Corollary 3.2]{KO}) that when $R$ is regular the above result holds for any 
$\A^{2}$-fibration over $R$.} 
\end{rem}

\begin{center}
{ ACKNOWLEDGEMENTS}
\end{center}
The authors thank S. M. Bhatwadekar for the current version of Theorem \ref{nonregular} and also for his useful comments and suggestions. The third author also acknowledges
Council of Scientific and Industrial Research (CSIR) for their grant.

\end{document}